\documentclass{amsart}
\usepackage{amssymb,amsmath}
\usepackage{amscd}
\headsep=1cm \numberwithin{equation}{section}
\newtheorem{thm}{Theorem}[section]
\newtheorem{cor}[thm]{Corollary}
\newtheorem{lem}[thm]{Lemma}
\newtheorem{prop}[thm]{Proposition}
\newtheorem{defn}[thm]{Definition}

\newtheorem{exam}[thm]{Example}
\newtheorem{rem}[thm]{Remark}

\newtheorem{no}[thm]{Notation}
\newcommand{\Ann}{\mbox{Ann}\,}
\newcommand{\coker}{\mbox{Coker}\,}

\newcommand{\Spec}{\mbox{Spec}\,}

\newcommand{\Ker}{\mbox{Ker}\,}
\newcommand{\Ass}{\mbox{Ass}\,}
\newcommand{\Assh}{\mbox{Assh}\,}
\newcommand{\Att}{\mbox{Att}\,}
\newcommand{\nCM}{\mbox{non--CM}\,}
\newcommand{\Supp}{\mbox{Supp}\,}

\newcommand{\depth}{\mbox{depth}\,}
\renewcommand{\dim}{\mbox{dim}\,}
\renewcommand{\Im}{\mbox{Im}\,}

\newcommand{\Min}{\mbox{Min}\,}

\newcommand{\h}{\mbox{ht}\,}

\renewcommand{\H}{\mbox{H}}
\newcommand{\V}{\mbox{V}}

\newcommand{\CM}{\mbox{CM}}

\newcommand{\fa}{\mathfrak{a}}

\newcommand{\fm}{\mathfrak{m}}
\newcommand{\fp}{\mathfrak{p}}
\newcommand{\fq}{\mathfrak{q}}
\newcommand{\fP}{\mathfrak{P}}
\newcommand{\fQ}{\mathfrak{Q}}
\newcommand{\fr}{\mathfrak{r}}

\newcommand{\C}{\mathcal{C}}
\newcommand{\mH}{\mathcal{H}}

\begin{document}

\bibliographystyle{amsplain}
\author{Mohammad T. Dibaei}
\address{ Mohammad T. Dibaei\\Faculty of Mathematical Sciences and Computer, Tarbiat Moallem
University, Tehran, Iran, and School of Mathematics, Institute for
Research in Fundamental Sciences (IPM) P. O. Box: 19395-5746,
Tehran, Iran.}

\email{dibaeimt@ipm.ir}

\author{Raheleh Jafari}
\address{Raheleh Jafari\\Faculty of Mathematical Sciences and Computer, Tarbiat Moallem
University, Tehran, Iran.}

\email{r.jafari@tmu.ac.ir}

\title[ Cohen-Macaulay Locus ]
{Cohen-Macaulay Loci of modules}

\keywords{Cohen-Macaulay locus, Cousin complexes, uniform
local cohomological annihilator, local cohomology.\\
The research of the first author was in part supported from IPM (No.
 87130117).}


\maketitle

%%% ----------------------------------------------------------------------
\begin{abstract} The Cohen-Macaulay locus of any finite module over
a noetherian local ring $A$ is studied and it is shown that it is a
Zariski-open subset of $\Spec A$ in certain cases. In this
connection, the rings whose formal fibres over certain prime ideals
are Cohen-Macaulay are studied.\end{abstract}
%%%-----------------------------------------------------------------------
%%------------------------------------------------------------------------
\section{Introduction}
Assume that $M$ is a module of finite dimension $d$ over a
noetherian ring $A$. For an ideal $\fa$ of $A$, denote $\H_\fa^i(M)$
as the $i$th local cohomology module of $M$ with respect to $\fa$ to
be the $i$th right derived functor of the section functor
$\Gamma_\fa(M):= \{x\in M: \fa^nx= 0$ for some positive integer $n
\}$.

 The Cohen-Macaulay locus of $M$ is denoted
by\begin{center} $\CM(M):=\{\fp\in\Spec A: M_\fp$ is Cohen-Macaulay
as $A_\fp$--module$\}$.\end{center} Trivially the Cohen-Macaulay
locus of a Cohen-Macaulay module is $\Spec A$ and  of a generalized
Cohen-Macaulay (g.CM for short) module $M$ over a local ring $(A,
\fm)$ (i.e. each local cohomology module $\H_\fm^i(M)$ has finite
length for all $i< d$) contains $\Spec A\setminus\{\fm\}$ (c.f.
\cite[Exercise 9.5.7]{BS}). In these cases $\CM(M)$  are
Zariski-open subset of $\Spec A$.

The topological property of the Cohen-Macaulay loci of modules is an
important tool. In \cite[Theorem 8.3]{K}, T. Kawasaki shows that
when the ring $A$ is catenary, the openness of $\CM(B)$ of any
finite $A$--algebra $B$ is a crucial assumption if one expects all
equidimensional finite $A$--module $M$ (i.e. $\dim_A(M)= \dim A/\fp$
for all minimal prime $\fp\in\Supp_A(M)$) have finite Cousin
complexes (see Notation 1.1).

The Cohen-Macaulay loci of modules have been studied by many
authors. Grothendieck in \cite{G} states that $\CM(M)$ is an open
subset of $\Spec A$ whenever $A$ is an excellent ring.  In
\cite{RS}, C. Rotthaus and L. $\c{S}$ega  study the Cohen-Macaulay
loci of graded modules over a noetherian homogeneous graded ring
$A=\bigoplus_{i\in\mathbb{N}}A_i$ considered as $A_0$--modules. In
\cite{H}, Grothendieck shows that $\CM(A)$ is open when $A$
possesses a dualizing complex. In \cite[Corollary 2.3]{D}, the first
author shows that $\CM(A)$ is open when $A$ is a local ring all of
its formal fibres are Cohen-Macaulay and satisfying the Serre
condition $(S_2)$, the condition which is superfluous (see Remark
3.4).

On the other hand  Sharp and Schenzel in \cite[Example 4.4]{SSc}
show that $M$
 is Cohen-Macualay if and only if the Cousin complex of $\mathcal{C}_A(M)$ is
 exact. Thus $\CM(M)= \Spec A\setminus \cup_{i\geq
 -1}\Supp_A(H^i(\mathcal{C}_A(M)))$, where $H^i(\mathcal{C}_A(M))$
 denotes the $i$th cohomology module of $\mathcal{C}_A(M)$.
  This fact simply implies that $\CM(M)$ is open whenever $\mathcal{C}_A(M)$ has finite cohomology modules.
   The authors in \cite[Theorem 2.7]{DJ} show that if $\mathcal{C}_A(M)$ has finite cohomology modules then $M$ has
    a uniform local cohomological annihilator (i.e. there exists an element
    $x\in A\setminus \underset{\fp\in\Min_A(M)}{\cup}\fp$
     such that $x\H_\fm^i(M)= 0$ for all $i< \dim_A(M)$ for all maximal ideals $\fm$ of $A$).

   Section 2 is devoted to study the set of attached primes,  $\Att_A(\H_\fm^i(M))$, of $i$th
   local cohomology module of $M$, over a local ring $(A, \fm)$,
   with respect to $\fm$. The set $\Att_A(\H_\fm^{d-1}(M))$ is
   determined, as an analogy of the formula
   $\Att_A(\H_\fm^d(M))=\Assh_A(M)$, under assumption that
   $\mathcal{C}_A(M)$ is finite (Proposition 2.3 and Corollary 2.6).

   In Section 3, we discuss about the Cohen-Macaulay locus of $M$.
  Lemma \ref{3.2} shows that in order to study openness of $\CM(M)$, in the case $A/0:_A M$ is catenary, it is
   sufficient to assume that $M$ is equidimensional. A new characterization
   of generalized Cohen-Macaulay rings is given in terms of uniform local cohomological annihilators (Corollary 3.12). As
   non--Cohen-Macaulay locus of $M$ is closed if and only if the set
   of its minimal members is finite, under some mild assumptions we prove
   that the set of minimal members of
   non--$\CM(M)$ is a subset of the union of $\Att_A(\H_\fm^i(M))$,
   $0\leq i\leq d$ and non--$\CM(A)$ (Theorem 3.13). As a result we
   show that
   the Cohen-Macaulay locus of any finite module over
a noetherian local ring $A$ is a Zariski-open subset of $\Spec A$ if
$A$ is catenary and the non--Cohen-Macaulay locus of $A$ is a finite
set (Corollary 3.14).

In Section 4, we prove that $M$ has a uniform local cohomological
annihilator if and only if $\widehat{M}$ is equidimensional and the
formal fibres over minimal members of $\Supp_A(M)$ are
Cohen-Macaulay (Theorem 4.2). As a result we obtain the following
interesting equivalence (Corollary 4.3):
\begin{itemize}
\item[(i)] $A$ is universally catenary  ring and all of its formal fibres are Cohen-Macaulay.
\item[(ii)] The Cousin complex $\mathcal{C}_{A}(A/\fp)$ is finite for all $\fp\in\Spec(A)$.
\item[(iii)]  $A/\fp$ has
a uniform local cohomological annihilator for all $\fp\in\Spec(A)$.
\end{itemize}
Moreover, the defining ideal of non-Cohen-Macaulay locus of $M$ is
given concretely under the finiteness of $\mathcal{C}_A(M)$
(Corollary 4.5).

%----------------------------------------------------------------
%1.1
\begin{no}\label{not2}
\emph{Let $M$ be an $A$--module and let $\mathfrak{H}= \{ H_l: l\geq
0\}$ be the family of subsets of $\Supp_A(M)$ with $H_l=
\{\fp\in\Supp_A(M): \dim_{A_\fp}(M_{\fp})\geq l\}$. The family
$\mathfrak{H}$ is called the $M$--height filtration of $\Supp_A(M)$.
Recall that the Cousin complex of $M$ is the complex
\begin{equation} \mathcal{C}_A(M):
0\overset{d^{-2}}{\longrightarrow}M^{-1}\overset{d^{-1}}{\longrightarrow}
M^0\overset{d^{0}}{\longrightarrow}M^1\overset{d^{1}}{\longrightarrow}
\cdots\overset{d^{l-1}}{\longrightarrow}M^l\overset{d^{l}}{\longrightarrow}
M^{l+1}\longrightarrow\cdots, \end{equation} where $M^{-1}= M$,
$M^l= \underset{\fp\in H_l\setminus H_{l+1}}{\oplus}(\coker
d^{l-2})_\fp$ for $l>-1$. The homomorphism $d^l: M^l\longrightarrow
M^{l+1}$ has the following property: for $m\in M^l$ and $\fp\in
H_l\setminus H_{l+1}$, the component of $d^l(m)$ in $(\coker
d^{l-1})_\fp$ is $\overline{m}/1$, where \  $\bar{} :
M^l\longrightarrow \coker d^{l-1}$ is the natural map (see \cite{S1}
for details). We choose the notations $$K^l:= \Ker d^l, D^l:= \Im
d^{l-1}, \mH^l_M:= K^l/D^l, l=-1, 0, \cdots.$$ We call the Cousin
complex $\mathcal{C}_A(M)$ finite whenever each $\mH^l_M$ is finite
as $A$--module. We have the following natural exact sequences:
\begin{equation} 0\longrightarrow M^{l-1}/K^{l-1}\longrightarrow
M^l\longrightarrow M^l/D^l\longrightarrow 0,
\end{equation}
\begin{equation} 0\longrightarrow \mH^{l-1}_M\longrightarrow
M^{l-1}/D^{l-1}\longrightarrow M^{l-1}/K^{l-1}\longrightarrow 0,
\end{equation}
for all $l\geq -1$.}
\end{no}
%---------------------------------------------------------------------
%---------------------------------------------------------------------
%Section 2

\section{Attached primes of local cohomology modules}
In this section we assume that ($A, \fm$) is local and that $M$ is a
finite $A$--module of dimension $d$. Although the main purpose of
this section is to find the main tool Corollary 2.6 to be used in
the proof of Lemma 3.11, the results are about the relations between
the set of attached primes of local cohomology modules $\H_\fm^i(M)$
and the cohomologies of the Cousin complex of $M$, which are
interesting on their own.

 The following result explains the situation
where all cohomology modules
 $\mH^i_M$ of the Cousin complex of $M$ are local cohomology
modules of $M$.
%----------------------------------------------------------------------

\begin{lem}\label{1.3}
%2.1
Assume that $(A, \fm)$ is a local ring and that $i$ is an integer
 with $0\leq i< d$. The following statements are equivalent.
\begin{enumerate}
\item[(i)] $\emph\dim_A(\mH^{j}_M)\leq 0$ for all $j$ with $-1\leq j< i$.
\item[(ii)]  $\emph\H^{j+1}_\fm(M)\cong\mH^{j}_M$ for all $j$ with $-1\leq j< i$.
\end{enumerate}
\end{lem}
\begin{proof}
Assume that $s$ is an integer such that $0\leq s<d$ and
$\dim_A(\mH_M^{s-1})\leq 0$. Consider the exact sequence (1.2) with
$l= s$ which gives the exact sequence
$$\H_\fm^{t-1}(M^{s})\longrightarrow\H_\fm^{t-1}
(M^{s}/D^{s})\longrightarrow
\H_\fm^{t}(M^{s-1}/K^{s-1})\longrightarrow\H_\fm^{t}(M^{s})$$ for
all integers $t$. As $s<d$, by \cite[Theorem]{S2}, we get
\begin{equation}\H_\fm^{t-1} (M^{s}/D^{s})\cong
\H_\fm^{t}(M^{s-1}/K^{s-1}).\end{equation}

 Next consider the
exact sequence
%$0\longrightarrow \mH_M^{s-1}\longrightarrow M^{s-1}/D^{s-1}\longrightarrow M^{s-1}/K^{s-1}\longrightarrow 0$
(1.3) with $l= s$ which gives the exact sequence
$$\H_\fm^{t}(\mH_M^{s-1})\longrightarrow\H_\fm^{t}
(M^{s-1}/D^{s-1})\longrightarrow
\H_\fm^{t}(M^{s-1}/K^{s-1})\longrightarrow\H_\fm^{t+1}(\mH_M^{s-1}).$$
Choosing $t>0$ in the above exact sequence we obtain that
\begin{equation}\H_\fm^{t} (M^{s-1}/D^{s-1})\cong \H_\fm^{t}(M^{s-1}/K^{s-1}).\end{equation}
As a consequence, from (2.1) and (2.2), we get
\begin{equation}\H_\fm^t(M^{s-1}/D^{s-1})\cong\H_\fm^{t-1}(M^{s}/D^{s})\end{equation}
for all $t> 0$.

(i)$\Rightarrow$ (ii). Let $-1\leq j< i$. By repeated use of (2.3),
we get $\H_\fm^{j+1}(M^{-1}/D^{-1})\cong \H_\fm^0(M^{j}/D^{j})$.
From the exact sequence (1.2) with $l= j+1$ we have
$\H_\fm^0(M^{j}/K^{j})= 0$ (because $j+1\leq i< d$ and
\cite[Theorem]{S2}). Hence the exact sequence (1.3) with $l= j+1$
implies that $ \H_\fm^0(M^j/D^j) \cong \H_\fm^0(\mH_M^j)\cong
\mH_M^j$. Therefore $\H_\fm^{j+1}(M)\cong \mH_M^j$.

(ii)$\Rightarrow$ (i) is clear.
\end{proof}

%------------------------------------------------------------------

 Recall that an artinian
$A$--module $N$ admits a reduced secondary representation
$N=N_1+\cdots+N_r$ so that $\fp_i= \Ann_A(N/N_i)$ is a prime ideal
of $A$, $ 1\leq i\leq r$. Denote the set of attached primes of $N$
by $\Att_A(N)=\{\fp_1, \ldots, \fp_r\}$. Each $N_i$ is
called a $\fp_i$--secondary module.\\

%----------------------------------------------------------
\begin{lem}
%2.2
Assume that $(A, \fm)$ is local and that $0\leq t< d$ is an integer
such that $\emph\dim_A(\mathcal{H}_M^i)\leq t-i-1$, for all $i\geq
-1$ ( $t= d-1$ is an example, see \cite[2.7 (vii)]{S1}). The
following statements hold true.
\begin{itemize}
\item[(i)]
$\emph\Att_A(\emph\H^t_\fm(M))\subseteq
\underset{i=-1,\ldots,t-1}{\bigcup}\emph\Att_A(\emph\H^{t-i-1}_\fm(\mH_M^i)).$

\item[(ii)] There is an epimorphism $\emph\H_\fm^t(M)\twoheadrightarrow\emph\H_\fm^0(\mH_M^{t-1})$.
\item[(iii)] Assume that $\mathcal{C}_A(M)$ is finite. Then $\mH_M^{t-1}$ is non-zero if and only if
$\fm\in\emph\Att_A(\emph\H^t_\fm(M))$.
\end{itemize}
\end{lem}
\begin{proof}
(i).  We prove by induction on $j$, $-1\leq j\leq t-1$, that
\begin{equation}
%(3.1)
\Att_A(\H_\fm^t(M))\subseteq\bigcup_{i\geq-1}^j
\Att_A(\H_\fm^{t-i-1}(\mH_M^i))\bigcup\Att_A(\H_\fm^{t-j-1}(M^{j}/K^{j})).\end{equation}
Due to  $\dim_A(\mH_M^i)\leq t-i-1$, the Grothendieck vanishing
theorem implies that $\H_\fm^{t-i}(\mH_M^i)= 0$. The exact sequence
(1.3) with $l= 0$ implies the exact sequence
$\H_\fm^t(\mH_M^{-1})\longrightarrow\H_\fm^t(M)\longrightarrow\H_\fm^t(M^{-1}/K^{-1})\longrightarrow
0$. Thus we get
$$\Att_A(\H_\fm^t(M))\subseteq\Att_A(\H_\fm^t(\mH_M^{-1}))\bigcup\Att_A(\H_\fm^t(M^{-1}/K^{-1})).$$
Assume that $-1\leq j< t-1$ and (2.4) holds. First note that (1.2)
with $l= j+1$  implies the exact sequence
$$\H_\fm^{t-j-2}(M^{j+1})\longrightarrow\H_\fm^{t-j-2}(M^{j+1}/D^{j+1})\longrightarrow\H_\fm^{t-j-1}(M^j/K^j)
\longrightarrow\H_\fm^{t-j-1}(M^{j+1}).$$ As $-1\leq j< t-1$,
$\H_\fm^{t-j-2}(M^{j+1})= 0$ and $\H_\fm^{t-j-1}(M^{j+1})= 0$ (c.f.
\cite [Theorem]{S2}). Therefore
\begin{equation}
%2.2
\H_\fm^{t-j-2}(M^{j+1}/D^{j+1})\cong\H_\fm^{t-j-1}(M^j/K^j).
\end{equation}
 On the other hand from the exact sequence (1.3) with $l= j+2$ we have the exact
sequence
\begin{equation}
%(2.6)
\H_\fm^{t-j-2}(\mH_M^{j+1})\longrightarrow\H_\fm^{t-j-2}(M^{j+1}/D^{j+1})
\longrightarrow\H_\fm^{t-j-2}(M^{j+1}/K^{j+1})\longrightarrow\H_\fm^{t-j-1}(\mH_M^{j+1}).\end{equation}
As $\H_\fm^{t-j-1}(\mH_M^{j+1})= 0$, (2.5) with the exact sequence
(2.6) imply that
\begin{equation}
%2.7
\begin{array}{ll}\Att_A(\H_\fm^{t-j-1}(M^{j}/K^{j}))&=
\Att_A(\H_\fm^{t-j-2}(M^{j+1}/D^{j+1}))\\&\subseteq
\Att_A(\H_\fm^{t-j-2}(\mH_M^{j+1}))\bigcup\Att_A(\H_\fm^{t-j-2}(M^{j+1}/K^{j+1})).\end{array}
\end{equation}
Now, (2.7) and (2.4) complete the induction argument. Thus we have
$$\Att_A(\H^t_\fm(M))\subseteq\underset{i= -1, 0, \cdots,
t-1}{\bigcup}\Att_A(\H_\fm^{t-i-1}(\mH_M^i))\bigcup\Att_A(\H_\fm^0(M^{t-1}/K^{t-1})).$$
On the other hand, considering the fact that $\H_\fm^0(M^t)= 0$, it
follows from the exact sequence (1.2) with $l= t$ that
$\H_\fm^0(M^{t-1}/K^{t-1})= 0$.

(ii). Consider the exact sequence (1.2) with $l= t-i$
 which imply the exact sequence
$$\H_\fm^{i-1}(M^{t-i})\longrightarrow\H_\fm^{i-1}(M^{t-i}/D^{t-i})\longrightarrow\H_\fm^{i}(M^{t-i-1}/K^{t-i-1})
\longrightarrow\H_\fm^{i}(M^{t-i}).$$ Taking $0\leq i\leq t$ we get
$0\leq t-i\leq t< d$ and so $\H_\fm^{i}(M^{t-i})= 0 =
\H_\fm^{i-1}(M^{t-i})$ (c.f. \cite[Theorem]{S2}). Therefore we have
isomorphisms
\begin{equation}
%2.6
\H_\fm^{i-1}(M^{t-i}/D^{t-i})\cong
\H_\fm^{i}(M^{t-i-1}/K^{t-i-1})\end{equation} for all $i$, $0\leq i
\leq t$.

Consider the exact sequence (1.3) with $l= t-i$ which induces the
exact sequence
$$\H_\fm^i(\mH_M^{t-i-1})\longrightarrow\H_\fm^i(M^{t-i-1}/D^{t-i-1})\longrightarrow
\H_\fm^i(M^{t-i-1}/K^{t-i-1})\longrightarrow\H_\fm^{i+1}(\mH_M^{t-i-1}).$$
As, by assumption $\dim_A(\mH_M^{t-i-1})\leq i$, we have
$\H_\fm^{i+1}(\mH_M^{t-i-1})= 0$ and so one obtains an epimorphism
\begin{equation}
%2.10
\H_\fm^i(M^{t-i-1}/D^{t-i-1})\twoheadrightarrow\H_\fm^i(M^{t-i-1}/K^{t-i-1}).
\end{equation}
By successive use of (2.9) and (2.8) one obtains an epimorphism

$$\H_\fm^t(M^{-1}/D^{-1})\twoheadrightarrow\H_\fm^0(M^{t-1}/D^{t-1}).
$$
 On the other hand, we have seen at the end of part (i)
that, we have
 $\H_\fm^0(M^{t-1}/K^{t-1})= 0$. Therefore, from  the exact sequence (1.3) with $l= t$,
 we get $\H_\fm^0(\mH_M^{t-1})\cong\H_\fm^0(M^{t-1}/D^{t-1})$
which
 results an epimorphism
 \begin{equation}
 %2.11
 \H_\fm^t(M)\twoheadrightarrow\H_\fm^0(\mH_M^{t-1}).\end{equation}

 (iii). Assume that $\mH_M^{t-1}\not =0$. As, by assumption
 $\dim_A(\mH_M^{t-1})\leq 0$,  we have
 $\H_\fm^0(\mH_M^{t-1})= \mH_M^{t-1}$ and so
 $\Att_A(\H_\fm^0(\mH_M^{t-1}))=\{\fm\}$. Now (2.10) implies
 that $\fm\in\Att_A(\H_\fm^t(M))$.

 Conversely, assume that $\fm\in\Att_A(\H_\fm^t(M))$. By part (i),
 $\fm\in\Att_A(\H_\fm^{t-i-1}(\mH_M^i))$ for some $i$, $-1\leq i\leq
 t-1$, and thus $\dim_A(\mH_M^i)\geq t-i-1$. As $\dim_A(\mH_M^i)\leq
 t-i-1$ we have equality $\dim_A(\mH_M^i)=
 t-i-1$. Note that $\mH_M^i$ is finite and so $\fm\in\Assh_A(\mH_M^i)$ (see \cite[Theorem 2.2]{MS})
 from which it follows
 that $t-i-1= 0$, i.e. $\mH_M^{t-1}\not=0.$
\end{proof}

It is well-known that $\Att_A(\H_\fm^d(M))= \Assh_A(M)$ (
\cite[Theorem 2.2]{MS}). The following result provides some
information about $\Att_A(\H_\fm^t(M))$ for certain $t$, in
particular for $t= d-1$.

%-----------------------------------------------------------------------------------
 \begin{prop}
 %2.3
 Assume that $(A, \fm)$ is local and that $\mathcal{C}_A(M)$ is finite. Let $t$, $0\leq t<
d$, be an integer such that $\emph\dim_A(\mH_M^i)\leq t-i-1$, for
all $i\geq -1$ (e.g. $t= d-1$, see \cite[(2.7) (vii)]{S1}). Then

 $$\emph\Att_A(\emph\H^{t}_\fm(M))=\bigcup^{t-1}_{i=-1}\{\fp\in\emph\Ass_A(\mH_M^i): \emph\dim(A/\fp)=t-i-1\}.$$
\end{prop}
\begin{proof}
Assume that $\fp\in\Att_A(\H_\fm^t(M))$. Then
$\fp\in\Att_A(\H_\fm^{t-i-1}(\mH_M^i))$ for some $i$, $-1\leq i\leq
t-1$ by Lemma 2.2(i). As $\dim_A(\mH_M^i)\leq t-i-1$, we have the
equality $\dim_A(\mH_M^i)= t-i-1$ and so $\fp\in\Ass_A(\mH_M^i)$ and
$\dim (A/\fp)= t-i-1$.

 Conversely, assume that $-1\leq
i_0\leq t-1$  and that $\fp\in \Ass_A(\mH_M^{i_0})$ such that
$\dim(A/\fp)=t-i_0-1$. Set $d':=\dim_{A_\fp}(M_{\fp})$ and $t':=
t-\dim (A/\fp)$. As $\C_A(M)$ is finite, $M$ is equidimensional and
$\Supp_A(M)$ is catenary (c.f. \cite[Corollary 2.12 and Corollary
2.15]{DJ}), we have $0\leq t'< d'$. Note that
$\mathcal{C}_{A_\fp}(M_\fp)\cong (\mathcal{C}_A(M))_\fp$ so that
$(\mH_M^j)_\fp\cong\mH_{M_\fp}^j$ for all $j$. As $\dim_{A_\fp}(
\mH_{M_\fp}^j)\leq \dim_A(\mH_M^j) -\dim (A/\fp)$ , we find that
$\dim_{A_\fp}(\mH^j_{M_\fp})\leq t'-j-1$ for all $j\geq -1$. On the
other hand $\mH_{M_\fp}^{t'-1}= (\mH_M^{i_0})_\fp\not =0$. By Lemma
2.2 (iii), replacing $M$ by $M_\fp$ implies that $\fp
A_\fp\in\Att_{A_\fp}(\H_{\fp A_\fp}^{t'}(M_\fp))$. Finally the Weak
General Shifted Localization Principle \cite[Exercise 11.3.8]{BS}
implies that $\fp\in\Att_A(\H_\fm^{t'+\dim A/\fp}(M))$ that is
$\fp\in\Att_A(\H_\fm^t(M))$.
\end{proof}

%------------------------------------------------------------------------------------
\begin{cor}
%2.4
 Assume that $(A, \fm)$ is a local ring, $M$ is a finite
$ A$--module and that $\mathcal{C}_A(M)$  is finite. Let $l< d$ be
an integer. The following statements are equivalent.
\begin{itemize}
\item[(i)] $\emph\H_\fm^j(M) =0$ for all $j$, $l<j< d$.
\item[(ii)] $\emph\dim_A(\mH_M^i)\leq l-i-1$ for all $i\geq -1.$
\end{itemize}
\end{cor}
\begin{proof} (i)$\Longrightarrow$(ii). We prove it by descending induction on $l$. For $l=
d-1$ we have nothing to prove (see also \cite[(2.7)(vii)]{S1}.
Assume that $l< d-1$. We have, by induction hypothesis, that
$\dim(A/\fp)\leq (l+1)-i-1$ for all $\fp\in\Supp_A(\mH_M^i)$ and for
all $i\geq -1$. If, for an ideal $\fp\in\Supp_A(\mH_M^i)$ and an
integer $i$, $\dim(A/\fp)=(l+1)-i-1$, then we get
$\fp\in\Ass_A(\mH_M^i)$ and so $\H_\fm^{l+1}(M)\not = 0$ by
Proposition 2.3, which contradicts the assumption. Therefore
$\dim(A/\fp)\not=(l+1)-i-1$ for any $\fp\in\Supp_A(\mH_M^i)$ and all
$i\geq -1$. That is $\dim_A(\mH_M^i)< (l+1)-i-1$ for all $i\geq -1$.
In other words, $\dim_A(\mH_M^i)\leq l-i-1$ for all $i\geq -1$.

(ii)$\Longrightarrow$(i). By descending induction on $l$. For $l=
d-1$ we have nothing to prove. Assume that $l<d-1$. As
$\dim_A(\mH_M^i)\leq l-i-1<(l+1)-i-1$ for all $i\geq -1$, we have,
by induction hypothesis, that $\H_\fm^j(M)=0$ for all $j$, $l+1<j<
d$. Moreover, Proposition 2.3 implies that $\Att_A(\H_\fm^{l+1}(M))$
is empty so that $\H_\fm^{l+1}(M)= 0$.
\end{proof}

The following result is now a clear conclusion of the above
corollary.

\begin{cor}
%2.5
 Assume that $(A, \fm)$ is a local ring, $M$ is a finite
$A$--module which is not Cohen-Macaulay and that $\mathcal{C}_A(M)$
is finite. Set $s= 1+ \sup\{\emph\dim_A(\mH_M^i)+i: i\geq -1\}$.
Then $\emph\H_\fm^s(M)\not= 0$ and  $\emph\H_\fm^i(M)= 0$ \emph{for
all} $i, s< i< d$.
\end{cor}

The following corollary gives us a non--vanishing criterion of
$\H_\fm^{d-1}(M)$ when $\mathcal{C}_A(M)$ is finite.
\begin{cor}
%2.6
Assume that $(A, \fm)$ is a local ring and that $M$ is a finite
$A$--module such that $\mathcal{C}_A(M)$ is finite. Then
\begin{itemize}
\item[(i)] $\emph\Att_A(\emph\H^{d-1}_\fm(M))=\bigcup^{d-2}_{i=-1}\{\fp\in\emph\Ass_A(\mH_M^i):
\emph\dim(A/\fp)=d-i-2\}.$
\item[(ii)] $\emph\H_\fm^{d-1}(M)\not =0$ if and only if
$\emph\dim_A(\mH_M^i)= d-i-2$ for some $i$, $-1\leq i\leq d-2$.
\end{itemize}
\end{cor}
\begin{proof}
It is clear by Proposition 2.3.
\end{proof}

%---------------------------------------------------------------------------
%--------------------------------------------------------------------------
% Section 3

\section{Cohen-Macaulay locus}
Throughout this section $A$ is a noetherian ring not necessarily
local and $M$ is a finite $A$--module. In the case $A$ is local, we
use $\widehat{M}$ as the completion of $M$ with respect to the
maximal ideal of $A$. The objective of this section is to study the
Cohen-Macaulay locus $\CM(M)=:\{\fp\in\Spec A: M_\fp \mbox{ is
Cohen-Macaulay } A_\fp\mbox{--module}\}$ of $M$. Our main goal is to
find out when it is a Zariski-open subset of $\Spec A$. We first
mention a remark for future references.
%------------------------------------------------------------------

\begin{rem}\label{3.1}
%3.1
For a finite $A$--module $M$ of finite dimension, if the Cousin
complex of $M$ is finite, then non-$\CM(M)=
\V(\underset{i}{\prod}(0:_A \mH_M^i))$
  so that $\CM(M)$
 is open, where $\mathcal{H}_M^i= H^i(\mathcal{C}_A(M))$.
\end{rem}
\begin{proof}
 It is clear, by \cite[Theorem 2.4]{S4} and \cite[Theorem 3.5]{S1},
 that $\CM(M)=\Spec(A)\setminus\underset{i\geq -1}{\cup}\Supp_A(\mH_M^i)$.

\end{proof}

%--------------------------------------------------------------------
The following lemma shows that the Cohen-Macaulay locus of a finite
module is open if it is true for certain submodules of $M$.

\begin{lem}\label{3.2}
%3.2
Let $A$ be a noetherian ring, let $M$ be a finite $A$--module, and let\\
 {\small $S=\{T\subseteq\emph\Min_A(M)$: there exists $\fq\in\emph\Supp_A(M)$ such that $\emph\h(\fq/\fp)$ is constant
for all $\fp\in T\}$.} For each $T\in S$, we assign a submodule
$M^T$ of $M$ with $\emph\Ass_A(M^T)= T$ and $\emph\Ass_A(M/M^T)=
\emph\Ass_A(M)\setminus T$. Then
$$\emph\CM(M)= \underset{T\in
S}{\bigcup}(\emph\CM(M^T)\setminus\underset{\fp\in\emph\Ass_A(M)\setminus
T}{\cup}\V(\fp)).$$
\end{lem}
\begin{proof} For each $T\in S$, it is clear that $$\Supp_A(M/M^T)=
 \underset{\fp\in\Ass_A(M)\setminus
T}{\cup}\V(\fp).$$ Let $\fq\in\CM(M)$ and set $T':= \{\fP\cap A:
\fP\in\Ass_{A_\fq}(M_\fq)\}$.
 As $M_\fq$ is Cohen-Macaulay, $\h(\fq/\fp)= \dim_A(M_\fq)$ for all $\fp\in T'$ and so
  $T'\in S$. We claim that $\fq\not\in\Supp_A(M/M^{T'})$. Assuming
  contrary, there is $\fp\in\Ass_A(M/M^{T'})$ such that $\fp\subseteq
  \fq$. Hence $\fp A_\fq\in\Ass_{A_\fq}(M_\fq))$ which implies that
   $\fp\in\ T'$. This contradicts with the fact that $\Ass_A(M/M^{T'})= \Ass_A(M)\setminus T'$.
 Therefore from the exact sequence
 \begin{equation}0\longrightarrow M^{T'}\longrightarrow
 M\longrightarrow M/M^{T'}\longrightarrow 0
 \end{equation}
  we get $(M^{T'})_\fq\cong M_\fq$
  so that $\fq\in\CM(M^{T'})$.

Conversely, assume that $T\in S$ and that
$\fq\in\CM(M^T)\setminus\underset{\fp\in\Ass_A (M)\setminus
T}{\cup}\V(\fp)$. That is $(M^T)_\fq$ is Cohen-Macaulay and
$\fq\not\in\Supp_A(M/M^T)$. Therefore $M_\fq$ is
   Cohen-Macaulay by (3.1), replacing $T$ by $T'$.
\end{proof}

Note that if $A/0:_A M$ is catenary, then each module $M^T$ in the
above lemma is an equidimensional $A$--module. Therefore one can
state the following remark.

\begin{rem} \emph{ If $A$ is catenary and $\CM(N)$ is open for all equidimensional
submodules $N$ of $M$, then $\CM(M)$ is open.}

\end{rem}

It is now a routine check to see that, over a local ring with
Cohen-Macaulay formal fibres, the Cohen-Macaulay locus of any finite
$A$--module is open.

%-----------------------------------------------------------------------------------------------------
\begin{rem}\label{3.7}
%3.5
 Assume that $A$ is a
noetherian local ring such that all of its
 formal fibres are Cohen-Macaulay. Then the Cohen-Macaulay locus of any finite $A$--module $M$ is
a Zariski-open subset of  $\emph\Spec A$.
\end{rem}
\begin{proof}
Equivalently, we prove that  $\Min($non--$\CM(M))$ is a finite set.
Choose $\fp\in\Min($non--$\CM(M))$ and let $\fQ$ be a minimal member
of the non--empty set $\{\fq\in\Supp_{\widehat{A}}(\widehat{M}) :
\fq\cap A=\fp\}$. Since the formal fibre of $A$ over $\fp$ is
Cohen--Macaulay, $\widehat{M}_\fQ$ is not Cohen-Macaulay by the
standard dimension and depth formulas.  On the other hand, for each
$\fq\in\Supp_{\widehat{A}}(\widehat{M})$ with $\fq\subset \fQ$ we
have $\fq\cap A\subset \fp$ and so $\widehat{M}_{\fq}$ is
Cohen-Macaulay again by the standard dimension and depth formulas.
Hence $\fQ\in\Min($non--$\CM(\widehat{M}))$.

Therefore, it is enough to show that $\Min($non--$\CM(\widehat{M}))$
is a finite set, equivalently $\CM(\widehat{M})$ is open. As
$\widehat{A}$ is catenary, we may assume that $\widehat{M}$ is
equidimensional $\widehat{A}$--module by Remark 3.3. Finally,
\cite[Theorem 5.5]{K} implies that
$\mathcal{C}_{\widehat{A}}(\widehat{M})$ has finite cohomologies and
so $\CM(\widehat{M})$ is open by Remark 3.1.
\end{proof}

 One may ask that, in Remark \ref{3.7}, what condition can be
 replaced instead of the Cohen-Macaulay--ness of all formal fibres.
 We are
going to introduce another such class of rings.

In \cite[Theorem 3.2]{Z}, it is shown that a finite dimensional ring
$A$ has a uniform local cohomological annihilator if and only if $A$
is locally equidimensional and $A/\fp$ has a uniform local
cohomological annihilator for all $\fp\in \Min(A)$. The module version of \cite[Theorem 3.2]{Z} is also true. \\

%---------------------------------------------------------------------------
\begin{prop} \label{ulc}
%3.5
Let $A$ be a ring and $M$ be a finite $A$--module of finite
dimension $d$. Then the following conditions are equivalent.
\begin{enumerate}
\item[(i)] $M$ has a uniform local cohomological annihilator.
\item[(ii)] $M$ is locally equidimensional and $A/\fp$ has a
uniform local cohomological annihilator for all $\fp\in
\emph\Min_A(M)$.
\end{enumerate}
\end{prop}
\begin{proof}
(i) $\Rightarrow $ (ii). By \cite[proposition 2.11]{DJ}, $M$ is
locally equidimensional. Let $\fp\in\Min_A(M)$. Then $\frac{A}{0:_A
M}/\frac{\fp}{0:_A M}\cong A/\fp$ has a uniform local cohomological
annihilator as in the proof of  \cite[Theorem 2.14]{DJ}.

 (ii) $\Rightarrow$ (i). This is a straightforward adaptation of
\cite[Theorem 3.2]{Z}.
\end{proof}
%-------------------------------------------------------------------------------
In the following, we give a characterization for a module to have a
uniform local cohomological annihilator in terms of the sets of
attached primes of local cohomology modules.

%3.7
\begin{lem}\label{3.6}
Assume that $(A, \fm)$ is local and that $M$ is a finite $A$--module
of dimension $d$. Then $M$ has a uniform local cohomological
annihilator if and only if
$\emph\Att_A(\emph\H^i_\fm(M))\cap\emph\Min_A(M)=\O$ for all
$i=0,\ldots,d -1$.
\end{lem}
\begin{proof}
Assume that $M$ has a uniform local cohomological annihilator.
 Therefore there is an element $x\in A\setminus
\underset{\fp\in\Min_A(M)}{\bigcup}\fp$ satisfying $x\H^i_\fm(M)=0$
for all $i=0,\ldots,d-1$. Thus, by \cite [Proposition 7.2.11]{BS},
  $x\in\underset{\fq\in\Att_A(\H^i_\fm(M))}{\bigcap}\fq$ for all $0\leq i\leq
  d-1$. Now the claim is clear.

Conversely,  we first note that if $\fp\in\Att_A(\H_\fm^i(M))$ then
$0:_A
  M\subseteq 0:_A \H_\fm^i(M)\subseteq \fp$ which gives
  $\fp\in\Supp_A(M)$. Therefore, by \cite
[Proposition 7.2.11]{BS} and prime avoidance, our assumption implies
that $\cap_{\tiny{i= 0}}^{\tiny{d-1}}(0:_A
\H_\fm^i(M))\not\subseteq\underset{\tiny{\fp\in\Min_A(M)}}{\cup}
\fp$.
\end{proof}
%----------------------------------------------------------------------------------

The following lemma is straightforward and we give a proof for
completeness.

%3.8
\begin{lem}\label{eq}
Assume that $A$ is a noetherian local ring and that $M$ is a finite
$A$--module.
 \begin{itemize}
 \item[(a)]If
  $\fQ\in\emph\Min_{\widehat{A}}(\widehat{M})$, then
  $\fQ\in\emph\Min_{\widehat{A}}(\widehat{A}/\fQ^{ce})$.
\item[(b)] If $A$ is universally catenary and $M$ is equidimensional, then
$\widehat{M}$ is equidimensional as $\widehat{A}$--module.
\end{itemize}
\end{lem}
\begin{proof}
(a). It follows that  $\fQ^c\in\Min_A(M)$ by the Going Down Theorem.
Assume that $\fQ'\in\Min_{\widehat{A}}(\widehat{A}/\fQ^{ce})$ such
that $\fQ'\subseteq\fQ$. The exact sequence $0\longrightarrow
A/\fQ^c\longrightarrow M$ implies the exact sequence
$0\longrightarrow \widehat{A}/\fQ^{ce}\longrightarrow \widehat{M}$.
Therefore $\fQ'\in\Ass_{\widehat{A}}(\widehat{M})$ and so
$\fQ'=\fQ$.

(b). Assume that $\fQ\in\Min_{\widehat{A}}(\widehat{M})$. By the
Going Down Theorem, $\fQ^c\in\Min_A(M)$ from which we have
$\dim_{\widehat{A}}(\widehat{A}/\fQ^{ce})= \dim_A(A/\fQ^c)=
\dim_A(M)$. As, by part (a),
$\fQ\in\Min_{\widehat{A}}(\widehat{A}/\fQ^{ce})$ and using the fact
that $A/\fQ^c$ is formally equidimensional, we get
$\dim(\widehat{A}/\fQ)=\dim_{\widehat{A}}({\widehat{A}/\fQ^{ce}})$
which implies that $\dim(\widehat{A}/\fQ)=\dim_A(M)$.
\end{proof}

%-----------------------------------------------------------------------------------

To prove our main result, Theorem 3.13, we need to know the
properties of a generalized Cohen-Macaulay module in terms of
certain prime ideals $\fp$ which $A/\fp$ has uniform local
cohomological annihilators. In this connection, we will prove that
``a local equidimensional ring $A$ is generalized Cohen-Macaulay if
and only if $A/\fp$ has uniform local cohomological annihilators for
all $\fp\in\Spec A$ and non--$\CM(M)\subseteq\{\fm\}$" (see
Corollary 3.12) which is interesting on its own.

Recall the following definition.
%---------------------------------------------------------------------
%1.3
\begin{defn} \emph{A finite module $M$ over a local ring $(A, \fm)$ with $d=\dim_A(M)$ is called
a} generalized Cohen-Macaulay \emph{(g.CM) module whenever
$\fm^n\H_\fm^i(M)= 0$ for some $n\in\mathbb{N}$ for all $i<d$}.
\emph{The module $M$ is called}  quasi-Buchsbaum \emph{whenever
$\fm\H_\fm^i(M)= 0$ for all $i<d$.}
\end{defn}
The following remark is easy but we bring it here for completeness
and future reference.

%-----------------------------------------------------------------------
%1.4
\begin{rem}\label{1.4}
Assume that $(A, \fm)$ is local.

\begin{enumerate}
\item[(a)] A finite $A$--module $M$ is g.CM if and only if all
cohomology modules of $\C_A(M)$ are of finite lengths.

\item[(b)]
     A finite $A$--module $M$ is quasi-Buchsbaum module if and only if
$\C_A(M)$ is finite and $\fm\mH^i_M=0$ for all $i$.
   \end{enumerate}
\end{rem}
\begin{proof}
(a).  Assume that $M$ is g.CM. By \cite[Exercise 9.5.7]{BS}, we have
$M_\fp$ is Cohen-Macaulay for all
$\fp\in\Supp_A(M)\setminus\{\fm\}$. So that
$\Supp_A(\mH^i_M)\subseteq \{\fm\}$ (\cite[Example 4.4]{SSc}) and,
by Lemma 2.1, the result follows. The converse is clear by Lemma
2.1.

(b). It is similar to (a).
\end{proof}

%-----------------------------------------------------------------------
%--------------------------------------------------------------------
%1.5
\begin{rem}\label{1.5}
Let $(A,\fm)$ be a g.CM local ring. Then $A/\fp$ has a uniform local
cohomological annihilator for all $\fp\in\emph\Spec(A)$. In
particular, any equidimensional $A$--module $M$ has a uniform local
cohomological annihilator.
\end{rem}
\begin{proof}
Let $\fp\in\Spec(A)$. Assume that $\h(\fp)=0$. As $A$ is g.CM, $A$
has a uniform local cohomological annihilator and thus $A/\fp$ has a
uniform local cohomological annihilator by  \cite[Theorem 3.2]{Z}.
Assume that $\h_M(\fp)=t>0$. There is a subset of system of
parameters $x_1,\ldots,x_t$ of $A$ contained in $\fp$. By
\cite[Exercise 9.5.8]{BS}, $A/(x_1,\ldots,x_t)$ is g.CM and so it
has a uniform local cohomological annihilator. In particular $A/\fp$
has a uniform local cohomological annihilator by \cite[Theorem
3.2]{Z}. The final part (immediately) follows from Proposition 3.5.
\end{proof}

The following result provides an answer to a partial converse of the
above statement.

%-----------------------------------------------------------------------------------
\begin{lem}\label{3.8}
%3.11
Assume that $(A, \fm)$ is a local ring such that $A/\fp$ has a
uniform local cohomological annihilator for all $\fp\in\emph\Spec
A$. For a finite $A$--module $M$ the following statements are
equivalent.

 \begin{itemize}
 \item[(i)] $M$ is  equidimensional $A$--module and
 non--$\CM(M)\subseteq\{\fm\}$.
 \item[(ii)] $M$ is a g.CM module.
 \end{itemize}
\end{lem}
\begin{proof} Note that $\widehat{A}/\fp\widehat{A}$ has a uniform
local cohomological annihilator and so $A/\fp$ is formally
equidimensional. Thus, by \cite[Theorem 31.7]{Ma}, $A$ is
universally catenary.

 (i)$\Rightarrow$ (ii). Since
$\H^i_\fm(M)\cong\H^i_\fm(M/\Gamma_\fm(M))$ for $i>0$, we may assume
that $\Gamma_\fm(M)=0$ and so $\fm\notin\Ass_A(M)$. As, for each
$\fp\in \Ass_A(M)$, $M_\fp$ is Cohen-Macaulay, so
$\Ass_A(M)=\Min_A(M)$.
% which means $M$ satisfies Serre condition ($S_1$).
 Thus $\mH^{-1}_M=0$ (see \cite[Example 4.4]{SSc}). Since
$M$ is equidimensional and $A/\fp$ has a uniform local cohomological
annihilator for all $\fp\in\Min_A(M)$, $M$ has a uniform local
cohomological annihilator by Proposition \ref{ulc} and so
$A/(0:_AM)$ is universally catenary by \cite[Theorem 2.14]{DJ}. As a
result, Lemma \ref{eq} implies that $\widehat{M}$ is
equidimensional. Hence $\C_{\widehat{A}}(\widehat{M})$ is finite by
\cite[Theorem 5.5]{K}.

Now, we prove the statement by using induction on $d=\dim_A(M)$. For
$d=2$, we have, by Corollary 2.6, that
$$\Att_{\widehat{A}}(\H^1_{\widehat{\fm}}(\widehat{M}))=\{\fp\in\Ass_{\widehat{A}}(\mH^{-1}_{\widehat{M}})
: \h_{\widehat{M}}(\fp)=1\}
\cup\{\fp\in\Ass_{\widehat{A}}(\mH^{0}_{\widehat{M}}) :
\h_{\widehat{M}}(\fp)=2\}.$$ If
$\fp\in\Ass_{\widehat{A}}(\mH^{-1}_{\widehat{M}})$ with
$\h_{\widehat{M}}(\fp)=1$,
 then $\fp\in\Ass_{\widehat{A}}(\widehat{M})$ and so  $\fp^c\in\Ass_A(M)=\Min_A(M)$.
 On the other hand, since
 $\fp\in\Att_{\widehat{A}}(\H_{\widehat{\fm}}^1(\widehat{M}))$,
 $\fp^c\in\Att_A(\H^1_\fm(M))$ by \cite[Lemma 2.1]{S5} which contradicts with Lemma \ref{3.6}.
 Hence $\Att_{\widehat{A}}(\H^1_{\widehat{\fm}}(\widehat{M}))\subseteq \{\widehat{\fm}\}$ and there exists an integer
 $n$ such that $\fm^n\H^1_{\widehat{\fm}}(\widehat{M})=0$ by \cite[7.2.11]{BS} and so
 $\H^1_{\fm}(M)\otimes_A\widehat{A}$ is finite $\widehat{A}$--module.
 It implies the first step of the induction.

 Now assume that $d>2$ and the statement holds up to $d-1$. Let $x$ be a uniform local cohomological
 annihilator of $M$.
 Since $\Min_A(M)=\Ass_A(M)$, $x$ is a nonzero divisor on $M$ by using its definition. On the other
 hand, as $A$ is catenary,  it is straightforward to
  see that $M/xM$
 satisfies the induction hypothesis for $d-1$. Therefore, $\H_\fm^i(M/xM)$ is finite for all $i< d-1$.
  The exact sequence
 $0\longrightarrow M\overset{x}{\longrightarrow} M\longrightarrow M/xM \longrightarrow 0$ implies the
 long exact sequence
 $$\cdots\longrightarrow\H^i_\fm(M) \overset{x.}{\longrightarrow} \H^i_\fm(M)\longrightarrow
  \H^i_\fm(M/xM)\longrightarrow \H^{i+1}_\fm(M)\overset{x.}{\longrightarrow} \H^{i+1}_\fm(M)
  \longrightarrow\cdots.$$
 Since $x\H^j_\fm(M)=0$ for $j<d$, we get the exact sequence
 $0\longrightarrow \H^i_\fm(M)\longrightarrow \H^i_\fm(M/xM)\longrightarrow \H^{i+1}_\fm(M)
 \longrightarrow0$, for $i=0,\ldots,d-2$. Now the result follows.

(ii)$\Rightarrow$(i) is clear.
 \end{proof}

%------------------------------------------------------------------------------------
Now we can state a criterion for an equidimensional local ring to be
a g.CM ring in terms of uniform local cohomological annihilators.

\begin{cor}\label{3.9}
%3.12
Assume that $A$ is an equidimensional noetherian local ring. The
following statements are equivalent.
\begin{itemize}
\item [(i)]$A$ is  g.CM.
\item [(ii)] $A/\fp$ has a uniform local cohomological annihilator for all $\fp\in\emph\Spec A$ and
non--$\emph\CM(A)\subseteq\{\fm\}$.
\end{itemize}
\end{cor}
\begin{proof}
(i)$\Rightarrow$(ii). We know that $A_\fp$ is Cohen-Macaulay for all
$\fp\in\Spec A\setminus\{\fm\}$, by \cite[Exercise 9.5.7]{BS}. The
rest is the subject of Remark \ref{1.5}.

(ii)$\Rightarrow$(i) is immediate from Lemma \ref{3.8}.
\end{proof}
%---------------------------------------------------------------------
We are now able to prove that any minimal element of non-$\CM(M)$ is
either an attached prime of $\H_\fm^i(M)$ for some $i$ or $A_\fp$ is
not a Cohen-Macaulay ring.
\begin{thm}\label{3.10}
Assume that $(A,\fm)$ is a catenary local ring and that $M$ is a
finite equidimensional $A$--module. Then

\centerline{$\emph\Min($non--$\emph\CM(M))\subseteq \underset{0\leq
i\leq\emph\dim_A(M)}{\cup}\emph\Att_A(\emph\H^i_\fm(M))\cup$non--$\emph\CM(A)$.}
\end{thm}
\begin{proof}

Choose $\fp\in\Min($non--$\CM(M))$. As $A$ is catenary and $M$ is
equidimensional, $M_\fp$ is also equidimensional as $A_\fp$--module.
Assume that $A_\fp$ is a Cohen-Macaulay ring.  For each
$\fq\in\Spec(A)$ with $\fq\subseteq\fp$, $A_\fp/\fq A_\fp$ has a
uniform local cohomological annihilator (c.f. \cite[Corollary
3.3]{Z}). Therefore, by Lemma \ref{3.8}, $M_\fp$ is a g.CM
$A_\fp$--module. As $M_\fp$ is not Cohen-Macaulay, $\H^i_{\fp
A_\fp}(M_\fp)\neq 0$ for some integer $i$, $i< \dim_{A_\fp}(M_\fp)$.
In particular, $\H^i_{\fp A_\fp}(M_\fp)$ is a non--zero finite
length $A_\fp$--module so that $\Att_{A_\fp}(\H^i_{\fp
A_\fp}(M_\fp))=\{\fp A_\fp\}$. By \cite[Exercise 11.3.8]{BS},
$\fp\in\Att_A(\H^{i+t}_\fm(M))$, where $t=\dim(A/\fp)$. Now the
result follows.
\end{proof}

%-----------------------------------------------------------------

\begin{cor}
Assume that $(A, \fm)$ is a catenary local ring and that the
non--$\emph\CM(A)$ is finite (e.g. $A$ may satisfy Serre condition
$(\emph{S}_{d-2})$, $d:=\emph\dim A$, and $\mathcal{C}_A(A)$ is
finite). Then the Cohen--Macaulay locus of $M$ is open for any
finite $A$--module $M$.
\end{cor}
\begin{proof} By Lemma \ref{3.2}, we may assume that $M$ is equidimensional.
Now \ref{3.10} implies that $\Min($non--$\CM(M))$ is a finite set.
In other words the non--$\CM(M)$ is Zariski closed subset of
$\Spec(A)$.
\end{proof}

In the following examples, we show that Remark 3.4 and Corollary
3.14 are significant. Example 3.15 gives a local ring $S$ with
Cohen-Macaulay formal fibres and the set non-$\CM(S)$ is infinite.
Example 3.16 presents a local ring $T$ which admits a
non--Cohen-Macaulay formal fibre with finite non-$\CM(T)$.

\begin{exam} \emph{Set $S= k[[X, Y, Z, U, V]]/(X)\cap(Y, Z)$, where $k$ is a
field. It is clear that $S$ is a local ring with Cohen-Macaulay
formal fibres. By \cite[Theorem 31.2]{Ma}, there are infinitely many
prime ideals $P$ of $k[[X, Y, Z, U, V]]$, $(X, Y, Z)\subset P\subset
(X, Y, Z, U, V)$. For any such prime ideal $P$, $S_{\overline{P}}$
is not equidimensional and so it is not Cohen-Macaulay. In other
words, non--$\CM(S)$ is infinite.}

\end{exam}
\begin{exam}
\emph{It is shown in \cite[Proposition 3.3]{FR} that there exists a
local integral domain $(R, \fm)$ of dimension 2 such that
$\widehat{R}= \mathbb{C}[[X, Y, Z]]/(Z^2, tZ)$, where $\mathbb{C}$
is the field of complex numbers and $t= X+ Y+ Y^2s$ for some
$s\in\mathbb{C}[[Y]]\setminus\mathbb{C}\{Y\}$. As $\Ass \widehat{R}=
\{(Z), (Z, t)\}$, $\widehat{R}$ does not satisfy $(S_1)$. Thus
$H^{-1}(\mathcal{C}(\widehat{R})) \neq 0$ while
$H^{-1}(\mathcal{C}(R))= 0$ (c.f.\cite[Example 4.4]{SSc}). By
\cite[Theorem 3.5]{P}, there exists a formal fibre of $R$ which is
not Cohen-Macaulay. As $R$ is an integral local domain, we have
non-$\CM(R)= \{\fm\}$.}
\end{exam}
%--------------------------------------------------------------------
%4
\section{Rings whose formal fibres are Cohen-Macaulay}

Throughout this section $(A, \fm)$ is a local ring and $M$ is a
finite $A$--module of dimension $d$. It is shown in section 3 that
over a ring $A$ with Cohen-Macaulay formal fibres the Cohen-Macaulay
locus of any finite $A$--module $M$ is a Zariski-open subset of
$\Spec A$ (Remark 3.4). This result motivates us to determine rings
whose formal fibres are Cohen-Macaulay. More precisely, we study the
affect of certain formal fibres being Cohen-Macaulay on the
structure of a module.

%----------------------------------------------------------------
In the following we first write a statement which summarizes the
results \cite[Corollary 2.3, Corollary 2.4, Lemma 2.5, Proposition
2.6, and Theorem 2.7]{DJ}. We denote $$\mathrm{a}(M):= \cap_{i=
0}^{d-1}(0 :_A \emph\H_\fm^i(M)).$$
\begin{thm}\label{A}
%4.1
Assume that $M$ is a finite $A$--module with dimension $d$. Then
$$\prod_{i= -1}^{d-1}(0 :_A \mathcal{H}_M^i)\subseteq\mathrm{a}(M).$$
Moreover, if $\mathcal{C}_A(M)$ is finite, then
$$\prod_{i= -1}^{d-1}(0 :_A \mathcal{H}_M^i)\not\subseteq\underset{\fp\in\emph\Min_A(M)} {\cup}\fp,$$ where
$\mathcal{H}_M^i$ denotes the $i$th cohomology module of the Cousin
complex  $\mathcal{C}_A(M)$.
\end{thm}
%----------------------------------------------------------------

%-----------------------------------------------------------
 The
following result gives a characterization of a finite module which
admits a uniform local cohomological annihilator in terms of a
certain set of formal fibres of the ground ring.

%--------------------------------------------------------------
\begin{thm}\label{B}
%4.2
 The following statements
are equivalent.
\begin{itemize}
\item[(i)] $\widehat{M}$ is equidimensional $\widehat{A}$--module
and all formal fibres of $A$ over minimal members of
$\emph\Supp_A(M)$ are Cohen-Macaulay.
\item[(ii)] $M$ has a uniform local cohomological annihilator.
\end{itemize}
\end{thm}
\begin{proof}
(i)$\Rightarrow$(ii). By \cite[Theorem 5.5]{K},
$\mathcal{C}_{\widehat{A}}(\widehat{M})$ is finite, which implies
that $\Min($non-$\CM(\widehat{M}))$ is a finite set (see Remark
3.1). Thus Theorem 4.1 implies that
\begin{equation}
\underset{\fq\in \nCM(\widehat{M})} {\cap}\fq \not\subseteq
\underset{\fq\in\Min_{\widehat{A}}(\widehat{M})} {\cup}\fq.
\end{equation}
We show that $(\underset{\fq\in \nCM(\widehat{M})} {\cap}\fq) \cap
A\not\subseteq \underset{\fp\in\Min_A(M)} {\cup}\fp$. Otherwise, by
the finiteness of $\Min(\nCM(\widehat{M}))$ there is
$\fp\in\Min_A(M)$ such that $\fp= \fq\cap A$ for some $\fq\in
\nCM(\widehat{M})$. Note that $A_\fp\longrightarrow
(\widehat{A})_\fq$ is a faithfully flat ring homomorphism and its
fibre ring over $\fp A_\fp$ is $((A_\fp/\fp A_\fp) \otimes_A
\widehat{A})_\fq$ which is Cohen-Macaulay by our assumption.
Therefore, by the standard dimension and depth formulas,  $M_\fp$ is
not Cohen-Macaulay. This contradicts with the fact that
$\dim_{A_\fp}(M_{\fp})= 0$.

Thus we take an element $r\in(\underset{\fq\in \nCM(\widehat{M})}
{\cap}\fq)\cap A\setminus \underset{\fp\in\Min_A(M)} {\cup}\fp$. By
Theorem 4.1, $r^n\H^i_{\widehat{\fm}}(\widehat{M})= 0$ for some
positive integer $n$ and for all $0\leq i< \dim_A(M)$. As
$\H_{\widehat{\fm}}^i(\widehat{M})\cong\H_\fm^i(M)$ the claim
follows.

(ii)$\Rightarrow$(i). As $M$ has a uniform local cohomological
annihilator, $\widehat{M}$ admits a uniform local cohomological
annihilator too so that $\widehat{M}$ is equidimensional. Assume
that $\fp\in\Min_A(M)$. We are going to show that $(A_\fp/\fp
A_\fp)\otimes_A \widehat{A}$ is Cohen-Macaulay. It is well known
that $(A_\fp/\fp A_\fp)\otimes_A \widehat{A}\cong
S^{-1}(\widehat{A}/\fp \widehat{A})$, where $S$ is the image of
$A\setminus \fp$ in $\widehat{A}$. Therefore we are going to show
that $(S^{-1}(\widehat{A}/\fp \widehat{A}))_{S^{-1}\fq}$ is
Cohen-Macaulay for all $\fq\in \Spec \widehat{A}$ with $S\cap\fq=
\emptyset$. It is enough to show that $(\widehat{A}/\fp
\widehat{A})_\fq$ is Cohen-Macaulay $\widehat{A}_\fq$--module. By
Proposition 3.5, $A/\fp $ has a uniform local cohomological
annihilator and hence $\widehat{A}/\fp \widehat{A}$ has a uniform
local cohomological annihilator which, in particular, implies that
$\widehat{A}/\fp \widehat{A}$ is equidimensional. Assume,
contrarily, $(\widehat{A}/\fp \widehat{A})_\fq$ is not
Cohen-Macaulay. We may assume that $\fq\in\Min(\nCM(\widehat{A}/\fp
\widehat{A}))$ and that $(\fq\cap A)\cap(A\setminus\fp)= \emptyset$.
In other words, $\nCM((\widehat{A}/\fp \widehat{A})_\fq)= \{\fq
\widehat{A}_\fq\}$ and $\fq\cap A= \fp$.

Let us replace $A$ and $M$ in Lemma 3.11 by $\widehat{A}_\fq$ and
$(\widehat{A}/\fp \widehat{A})_\fq$, respectively. As
$\mathcal{C}_{\widehat{A}}(\widehat{A}/\fq')$ is finite (c.f.
\cite[Theorem 5.5]{K}) for all $\fq'\in\Spec \widehat{A}$ so that
$\mathcal{C}_{\widehat{A}_\fq}(\widehat{A}_\fq/\fq'\widehat{A}_\fq)$
is finite. Therefore $\widehat{A}_\fq/\fq'\widehat{A}_\fq$ has a
uniform local cohomological annihilator as $\widehat{A}_\fq$--module
 for all $\fq'\in\Spec \widehat{A}$ with $\fq'\subseteq\fq$ (c.f. \cite[Theorem 2.7]{DJ}). As $(\widehat{A}/\fp
\widehat{A})_\fq$ is equidimensional, we can apply Lemma 3.11 to
deduce that $(\widehat{A}/\fp \widehat{A})_\fq$ is a g.CM as
$\widehat{A}_\fq$--module. In particular,
$\H_{\fq\widehat{A}_\fq}^i((\widehat{A}/\fp \widehat{A})_\fq))$ is a
non--zero  $\widehat{A}_\fq$--module of finite length for some
$i<\dim(\widehat{A}/\fp \widehat{A})_\fq$ for which we get
$\Att_{\widehat{A}_\fq}(\H_{\fq\widehat{A}_\fq}^i((\widehat{A}/\fp
\widehat{A})_\fq))= \{\fq\widehat{A}_\fq\}$. Now, \cite[Exercise
11.3.8]{BS} implies that
$\fq\in\Att_{\widehat{A}}(\H_{\widehat{\fm}}^j(\widehat{A}/\fp\widehat{A}))$
for some $j< \dim A/\fp$ which gives $\fp= \fq\cap
A\in\Att_A(\H_\fm^j(A/\fp))$. This contradicts with Lemma 3.6.
\end{proof}

The following result shows that if all formal fibres of a ring $A$
are Cohen-Macaulay then $A$ is universally catenary if and only
$\mathcal{C}_A(A/\fp)$ is finite for all $\fp\in\Spec A$.
%----------------------------------------------------------
\begin{cor} \label{C}
%4.3
The following statements are equivalent.
\begin{itemize}
\item[(i)] $A$ is universally catenary  ring and all of its formal fibres are Cohen-Macaulay.
\item[(ii)] The Cousin complex $\mathcal{C}_{A}(A/\fp)$ is finite for all $\fp\in\emph\Spec(A)$.
\item[(iii)]  $A/\fp$ has
a uniform local cohomological annihilator for all
$\fp\in\emph\Spec(A)$.
\end{itemize}
\end{cor}
\begin{proof}
(i)$\Rightarrow$(ii) is clear by \cite[Theorem 5.5]{K}.

(ii)$\Rightarrow$(iii) is clear by Theorem 4.1.

(iii)$\Rightarrow$(i). To see $A$ is universally catenary we may
argue as the beginning of the proof of Lemma 3.11. The rest is clear
by Theorem 4.2.

\end{proof}

%--------------------------------------------------------------
Our final goal is to investigate the relationship between
non-$\CM(M)$ and $\V(\mathrm{a}(M))$. We will see that non-$\CM(M)=
\V(\mathrm{a}(M))$ whenever $\mathcal{C}_A(M)$ is finite. However,
the following statement holds true whenever $\mathcal{C}_A(M)$ is
not finite.
%------------------------------------------------------------------------

\begin{rem}\label{E}
%4.4
  non--$\CM(M)\subseteq \V(\mathrm{a}(M))$.
\end{rem}
\begin{proof} Let $d:=\dim M$. Choose  $\fp\in $non--$\CM(M)$ and assume contrarily
that $\mathrm{a}(M)\nsubseteq \fp$. Then by \cite[Theorem 9.3.5]{BS}
$$d=f^{\mathrm{a}(M)}_\fm(M)\leq
\lambda^{\mathrm{a}(M)}_\fm(M)\leq \depth M_\fp+\dim A/\fp \leq \dim
M_\fp+\dim A/\fp\leq d,$$ where $$f^{\mathrm{a}(M)}_\fm(M)=
\inf\{i\in\Bbb{N}:
\mathrm{a}(M)\not\subseteq\sqrt{(0:_A\H^i_\fm(M))}\}$$ and
$$\lambda^{\mathrm{a}(M)}_\fm(M)=\inf\{\depth M_\fp +\dim A/\fp:
\fp\in\Spec A\setminus \V(\mathrm{a}(M))\}.$$ Hence $\depth
M_\fp=\dim M_\fp$ that is contradiction.
\end{proof}

%----------------------------------------------------------------------------
\begin{cor}\label{F}
%4.6
Assume that $M$ is a finite $A$--module of dimension $d$ and that
$\mathcal{C}_A(M)$ is finite.  Then

\centerline{$\V(\prod_{i= -1}^{d-1}(0 :_A
\mathcal{H}_M^i))=$non--$\CM(M)= \V(\mathrm{a}(M))$.}
\end{cor}
\begin{proof}
The first equality is in Remark 3.1. The second inequality is clear
by Remark \ref{E} and  Theorem \ref{A}.
\end{proof}

%.......................................................................
For our final goal, i.e. to characterize those modules $M$ with
non--$\CM(M)= \V(\mathrm{a}(M))$, we need to characterize those
prime ideals $\fp$ such that $A/\fp$ has a uniform local
cohomological annihilator (see also Corollary 4.3).

\begin{prop}\label{G}
%4.7
Assume that $\fp\in\emph\Spec A$. A necessary and sufficient
condition for  $A/\fp$ to have a uniform local cohomological
annihilator is that there exists an equidimensional $A$--module $M$
such that
 $\fp\in\emph\Supp_A(M)\setminus\V(\mathrm{a}(M))$.
\end{prop}
\begin{proof}
The necessary condition is clear by taking $M:=A/\fp$. For the
converse, assume that there is an equidimensional $A$--module $M$
such that $\fp\in\Supp_A(M)\setminus\V(\mathrm{a}(M))$. We prove the
claim by induction on $h:= \h_M(\fp)$. When $h=0$, we have
$\fp\in\Min_A(M)$. Choose a submodule $N$ of $M$ with $\Ass_A(N)=
\{\fp\}$ and $\Ass_A(M/N)=\Ass_A(M)\setminus\{\fp\}$. It is clear
that $(M/N)_\fp =0$ so that $r(M/N)= 0$ for some $r\in
A\setminus\fp$. On the other hand the fact that
$\mathrm{a}(M)\not\subseteq\fp$ implies that there is $s\in
A\setminus\fp$ such that $s\H_\fm^i(M)= 0$ for all $i<\dim_A(M)$.
The exact sequence
$\H_\fm^{i-1}(M/N)\longrightarrow\H_\fm^i(N)\longrightarrow\H_\fm^i(M)$
implies $rs\H_\fm^i(N)= 0$ for all $i<\dim_A(N)$. As
$\fp\in\Min_A(N)$, $A/\fp$ has a uniform local cohomological
annihilator by Proposition \ref{ulc}.

Now assume that $h>0$. For any $\fq\in\Supp_A(M)$ with
$\fq\subset\fp$ we have $\fq\not\in\V(\mathrm{a}(M))$ so that
$A/\fq$ has a uniform local annihilator by induction hypothesis. As
$\fp\not\in\V(\mathrm{a}(M))$, $M_\fp$ is Cohen--Macaulay by Remark
\ref{E}. Choose a submodule $K$ of $M$ with $\Ass_A(K)= \Min_A(M)$
and $\Ass_A(M/K)= \Ass_A(M)\setminus\Min_A(M)$. If
$\fp\in\Supp_A(M/K)$ then there is $\fq\in\Ass_A(M/K)$ with
$\fq\subseteq\fp$. Therefore $\fq\in\Ass_A(M)$ and
$\fq\not\in\Min_A(M)$. As $M_\fp$ is Cohen-Macaulay, so is $M_\fq$
which gives $\fq\in\Min_A(M)$, which is a contradiction. Hence we
have $\fp\not\in\Supp_A(M/K)$ which yields $r(M/K)=0$ for some $r\in
A\setminus\fp$ and so, by applying local cohomology to the exact
sequence $0\longrightarrow K\longrightarrow M\longrightarrow
M/K\longrightarrow 0$, it follows that
$\mathrm{a}(K)\not\subseteq\fp$. As $M_\fp\cong K_\fp$, $K_\fp$ is
Cohen--Macaulay and $\h_K(\fp)>0$, there is $x\in\fp$ which is
non--zero divisor on $K$. The exact sequence $0\longrightarrow
K\overset{x}{\longrightarrow} K\longrightarrow K/xK\longrightarrow
0$ implies that $\mathrm{a}(K)^2\subseteq\mathrm{a}(K/xK)$ which
implies that $\mathrm{a}(K/xK)\not\subseteq\fp$. As
$\h_{K/xK}(\fp)<h$, $A/\fp$ has a uniform local cohomological
annihilator by the induction hypothesis.
\end{proof}

In Corollary \ref{F}, it is shown that, for a finite $A$--module
$M$, non--$\CM(M)= \V(\mathrm{a}(M))$ whenever $\mathcal{C}_A(M)$ is
finite. In the following we characterize those modules satisfying
this equality without assuming that the Cousin complex of $M$ to be
finite.

%--------------------------------------------------------------

\begin{thm}
%4.7
For an equidimensional $A$--module $M$, the following statements are
equivalent.
\begin{itemize}
\item[(i)] $A/\fq$ has a uniform local cohomological annihilator for
all $\fq\in\emph\CM(M)$.
\item[(i$'$)] $\widehat{A}/\fq\widehat{A}$ is equidimensional
$\widehat{A}$--module and the formal fibre ring $(A_\fq /\fq
A_\fq)\otimes_A \widehat{A}$ is Cohen-Macaulay for all
$\fq\in\emph\CM(M)$.
\item[(ii)] non--$\emph\CM(M)= \V(\mathrm{a}(M))$.
\item[(iii)] non--$\emph\CM(M)\supseteq \V(\mathrm{a}(M))$.
\end{itemize}
\end{thm}
\begin{proof}
The equivalence of (i) and (i$'$) is the subject of Theorem \ref{B}.

(i)$\Longrightarrow$(ii). The inclusion
non-$\CM(M)\subseteq\V(\mathrm{a}(M))$ is clear by Remark \ref{E}.

Now assume that $\fp\supseteq\mathrm{a}(M)$. Thus there is an
integer $i$, $0\leq i< d$, such that $\fp\supseteq 0:_A
\H_\fm^i(M)$. There is
$\fQ\in\Att_{\widehat{A}}(\H_{\widehat{\fm}}^i(\widehat{M}))$ with
$\fq:=A\cap\fQ\in\Att_A(\H_\fm^i(M))$  and $\fp\supseteq\fq$. To
show $\fp\in$non--$\CM(M)$ it is enough to show that
$\fq\in$non--$\CM(M)$. Assuming contrarily, $\fq\in\CM(M)$, $A/\fq$
has a uniform local cohomological annihilator by our assumption and
so the formal fibre $k(\fq)\otimes_A \widehat{A}$ is Cohen-Macaulay,
by Theorem \ref{B}. As the map $A_\fq\longrightarrow
\widehat{A}_\fQ$ is faithfully flat ring homomorphism, we find that
$k(\fq)\otimes_{A_\fq}\widehat{A}_\fQ$ is also Cohen-Macaulay.
Therefore the standard dimension and depth formulas, applied to the
faithfully flat extension $A_\fq\longrightarrow \widehat{A}_\fQ$,
implies that $\widehat{M}_\fQ$ is Cohen-Macaulay. On the other hand,
$A/\fr$ has a uniform local cohomological annihilator for all
$\fr\in\Min_A(M)$ (simply because in this case $M_\fr$ has zero
dimension and so $\fr\in\CM(M)$). Thus, by Proposition \ref{ulc},
$M$ has a uniform local cohomological annihilator and so does
$\widehat{M}$. Therefore, by \cite[Corollary 2.12]{DJ},
$\widehat{M}$ is equidimensional. Thus \cite[Theorem 5.5]{K} implies
that the Cousin complex $\mathcal{C}_{\widehat{A}}(\widehat{M})$ is
finite. As
$\fQ\in\Att_{\widehat{A}}(\H_{\widehat{\fm}}^i(\widehat{M}))$, we
have, by Corollary \ref{F}, $\fQ\in$non--$\CM(\widehat{M})$. This is
a contradiction.

(iii)$\Longrightarrow$(i). Assume that $\fq\in\CM(M)$ so that
$\fq\not\supseteq\mathrm{a}(M)$ by our assumption. Now Proposition
\ref{G} implies that $A/\fq$ has a uniform local cohomological
annihilator.
\end{proof}
{\bf Acknowledgment.} The authors would like to express their
gratitude to the referee for giving several comments and
suggestions.
%--------------------------------------------------------------

\end{document}